\documentclass[11pt]{article}
\usepackage{amsfonts, amsmath, amssymb, amsthm, tabularx, booktabs, multirow, enumitem}
\usepackage[colorlinks=true,citecolor=blue,urlcolor=blue,linkcolor=blue,bookmarksopen=true]{hyperref}
\usepackage{aliascnt,cleveref}

\date{}

\author{Boris Bukh}
\title{An improvement of the Beck--Fiala theorem}

\newtheorem*{maintheorem}{Theorem}

\newaliascnt{lemma}{theorem}
\newtheorem{lemma}[lemma]{Lemma}
\aliascntresetthe{lemma}
\crefname{lemma}{lemma}{lemmas}

\newaliascnt{proposition}{theorem}
\newtheorem{proposition}[proposition]{Proposition}
\aliascntresetthe{proposition}
\crefname{proposition}{proposition}{propositions}

\newcommand*{\abs}[1]{\lvert #1\rvert}                
\newcommand*{\eqdef}{\stackrel{\text{\tiny{def}}}{=}} 
\newcommand*{\R}{\mathbb{R}}                          
\newcommand*{\F}{\mathcal{F}}                         
\newcommand*{\B}{\mathcal{B}}                         
\newcommand*{\G}{\mathcal{G}}                         
\newcommand*{\C}{\mathcal{C}}                         
\newcommand*{\Cb}{\mathcal{D}}                        
\newcommand*{\Nb}{\mathcal{N}}                        
\newcommand*{\Pb}{\mathcal{P}}                        

\newcommand*{\eps}{\varepsilon}

\DeclareMathOperator{\disc}{disc}                     
\DeclareMathOperator{\Fr}{Fr}                         
\DeclareMathOperator{\Fl}{Fl}                         
\DeclareMathOperator{\Sz}{Sz}                         
\DeclareMathOperator{\Tw}{Tw}                         
\DeclareMathOperator{\Ch}{Ch}                         
\DeclareMathOperator{\sign}{sign}                     

\makeatletter
\def\Th{\@ifnextchar({\qopname\relax o{Th}}{\qopname\relax o{Th}_}}
\makeatother

\makeatletter
\newcommand*{\wackyenum}[1]{%
  \expandafter\@wackyenum\csname c@#1\endcsname%
}

\newcommand*{\@wackyenum}[1]{%
  $\ifcase#1\or1\or2\or3\or4\or5\or6\or7\or8^+\or8^-\or9\or10\or11^+\or11^-\or12^-\or12^+\or13^-\or13^+\or14%
    \else\@ctrerr\fi$%
}

\AddEnumerateCounter{\wackyenum}{\@wackyenum}{53.13}
\makeatother

\newlist{invlist}{enumerate}{2}
\setlist[invlist,1]{label=\textit{Invariant \wackyenum*:},ref=\wackyenum*,itemindent=10ex,labelwidth=13.5ex, align=left} 

\crefname{invlisti}{invariant}{invariants}
\crefname{stepno}{step}{steps}

\def\verif#1{\paragraph{Verification of \cref{#1}:}}

\newcounter{stepno}
\newif\ifsteppoint
\steppointtrue

\def\step#1#2#3{%
  \parskip=0pt plus 1pt\parindent0pt\hangafter1%
  \everypar={\parindent0pt\hangindent20pt\hangafter1}%
  \vskip 1.6ex plus 0.24ex minus 0.15ex%
  \refstepcounter{stepno}\label{#3}\textsf{Step \thestepno\ (#1):}\par\penalty40%
  \textsc{Execution condition: }#2\ifsteppoint.\else\steppointtrue\fi\par\penalty40%
  \textsc{Step description:}%
  \parskip=0.7ex plus 0.05ex minus 0.08ex\everypar={\hangafter0\hangindent20pt\parindent0pt}}

\def\stepalt#1#2#3{%
  \parskip=0pt plus 1pt\parindent0pt\hangafter1%
  \everypar={\parindent20pt\hangindent20pt\advance\hangindent3.4ex\hangafter1}%
  \vskip 1.6ex plus 0.24ex minus 0.15ex%
  \refstepcounter{stepno}\textsf{Step \thestepno\ (#1):}\par\penalty40%
  \noindent\textsc{Execution condition: } One of the following two conditions holds:\par%
  \makebox[3.4ex][r]{$+$) }#2.\par%
  \makebox[3.4ex][r]{$-$) }#3.\par\penalty40%
\everypar={\parindent0pt\hangindent20pt\hangafter1}%
  \noindent\textsc{Step description:}%
  \parskip=0.7ex plus 0.05ex minus 0.08ex\everypar={\hangafter0\hangindent20pt\parindent0pt}}

\makeatletter
\def\refconst#1{%
inequality (\hyperref[eq:const_inequalities]{4#1})}
\makeatother

\def\mathrlap{\mathpalette\mathrlapinternal}
\def\mathrlapinternal#1#2{%
\rlap{$\mathsurround=0pt#1{#2}$}}

\begin{document}
\maketitle

\begin{abstract}
In 1981 Beck and Fiala proved an upper bound for the discrepancy of 
a set system of degree~$d$ that is independent of the size of the 
ground set. In the intervening years the bound has been decreased
from $2d-2$ to $2d-4$. We improve the bound to $2d-\log^* d$.
\\[1ex]\textsc{MSC classes:} 05D05, 11K38, 05C15
\end{abstract}

\section{Introduction}
Let $X$ be a finite set, and let $\F$ be a family of subsets of $X$. A \emph{two-coloring} of $X$ is 
a function $\chi\colon X\to \{-1,+1\}$. For $S\subset X$ we define $\chi(S)\eqdef \sum_{x\in S} \chi(x)$.
The \emph{discrepancy} of a coloring $\chi$ is
\[
  \disc \chi\eqdef\max_{S\in\F} \,\abs{\chi(S)}.
\]
The discrepancy of $\F$ is then defined as the discrepancy of an optimal coloring,
\[
  \disc \F\eqdef\min_{\chi} \disc \chi.
\]

The \emph{degree} of $x\in X$ in the family $\F$ is the number of sets that contain $x$.
Over 30 years ago Beck and Fiala \cite{beck_fiala} proved that
if the maximum degree of vertices in $\F$ is $d$, then $\disc \F\leq 2d-2$.
The remarkable feature of the result is that it depends neither on the number of sets in $\F$
nor on the size of $X$. If dependence on these quantities is permitted, the best result
is due to Banaszczyk \cite[Theorem~2]{banaszczyk}, and asserts that $\disc \F\leq c\sqrt {d \log \abs{X}}$.
Let $f(d)\eqdef \max \disc \F$, where the maximum is taken over all set families of degree at most $d$.
In \cite{beck_fiala} Beck and Fiala conjectured that $f(d)=O(\sqrt{d})$. If true,
the conjecture would be a strengthening of Spencer's six deviation theorem \cite{spencer_sixdev}.
A related, but stronger conjecture was made by J\'anos Koml\'os \cite[p.~680]{spencer_sixdev}.
A relaxation of Koml\'os's conjecture to vector-valued colorings has been established by Nikolov \cite{nikolov_relax}.
For a general overview of discrepancy theory see books \cite{matousek_book,chazelle_book}.

The original Beck--Fiala bound has been improved twice. First, Bednarchak and Helm \cite{bednarchak_helm} proved that
$f(d)\leq 2d-3$ for $d\geq 3$. Then Helm \cite{helm} claimed\footnote{The author has been unable to understand Helm's proof, 
or to reach Martin Helm for clarification.} that $f(d)\leq 2d-4$ for all sufficiently large $d$.
In this paper, we improve the bound by a function growing to infinity with $d$:
\begin{maintheorem}
For all sufficiently large $d$ we have
\[
  f(d)\leq 2d-\log^* d.
\]
\end{maintheorem}
Here $\log^*x \eqdef \min \{ t : \log^{(t)} x\leq 1\}$, where $\log^{(1)} x=\log x$ and $\log^{(t+1)}x=\log^{(t)}\log x$,
and the logarithms are to the base $2$.

\section{Proof ideas}
Like the proofs in \cite{beck_fiala,bednarchak_helm,helm}, our proof uses the method of floating colors. 
In this section we present the method informally, and explain the main difficulty in its application.
We present the three main ideas of our proof, and how these ideas address the difficulty.
In following sections we give the proof in full.

A \emph{floating coloring} is a function $\chi\colon X\to [-1,+1]$. The value $\chi(x)$ is the ``color''
of $x$ according to $\chi$. If $-1<\chi(x)<+1$, we consider the 
color of $x$ ``floating'', whereas $\chi(x)\in \{-1,+1\}$ means that the color of $x$ is ``frozen''.
Once frozen, the elements never change their color again. All the floating elements eventually
turn frozen, giving a genuine two-coloring of $X$. Our goal is to ensure that the discrepancy of that coloring 
does not exceed $2d-\Delta$, where $2d-\Delta$ is the desired bound on $f(d)$.

The purpose of evolving coloring $\chi$ gradually is to focus on the dangerous sets.
A set $S\in \F$ is \emph{dangerous} if there is a way to freeze the floating elements to make the
discrepancy exceed $2d-\Delta$. Only dangerous sets matter in the subsequent evolution of $\chi$. In
the argument of Beck and Fiala, an invariant is maintained: for each dangerous set $S\in\F$
we have $\chi(S)=0$. \emph{Size of a dangerous set} is the number of floating elements in the set.
It is easily shown that if the average size of dangerous sets exceeds $d$,
then the number of floating colors exceeds the number of dangerous sets; hence, there 
is a way to perturb the floating colors, in a manner that preserves the invariant. 
The coloring is perturbed until one of the floating elements becomes frozen, and the process 
repeats.

The dangerous sets of size at most $d$ thus pose a natural problem. If not for them,
the process would never stop, and the result is a coloring of discrepancy at most $2d-\Delta$.
Let us call these particularly troublesome dangerous sets \emph{nasty}. 
The sum of the floating elements of a nasty set is nearly $\pm d$, for otherwise
the floating elements do not have enough ``room'' to change much. Hence,
most elements of a nasty set are close to $\pm 1$.
The \emph{first idea} thus is to forcibly round elements $x$ such that $\abs{\chi(x)}>1-\alpha$,
where $\alpha$ is a small number. Forcible rounding of only $O(\Delta)$ elements in a nasty set is enough to render
the set benign. 

Forcible rounding introduces an ``error'' of at most $\alpha$ into the invariant $\chi(S)=0$.
Therefore, forcible rounding is tolerable only if the rounded element is not contained in any very large set.
The \emph{second idea} consists in noting that if there are nasty sets, but no elements in them
can be forcibly rounded, then the nasty sets must be highly overlapping. Indeed,
it is possible to perturb $\chi$ if the average size of sets exceeds $d$,
and the large sets that prevent forcible rounding contribute a lot to this average.
Making this idea precise requires a charging argument, whose details are in \cref{sec:newcohort}.

We take advantage of the overlap in the nasty sets by selecting an element $b$
that is common to many nasty sets, and singling out the nasty sets containing
$b$ into a separate \emph{cohort}. We call $b$ the cohort's \emph{banner}.
Since perturbation of $\chi$ might get stuck more than once, we might create
several cohorts over the course of our algorithm. We shall treat each cohort as a 
fully autonomous set system.  Thus it will be subject to its own invariants,
and will impose its own linear conditions in the perturbation step.

Since cohorts consist of nasty sets, the average set size in a cohort
is less than $d$, and so for the perturbation step to be possible, a cohort
cannot impose as many linear conditions as there are sets in a cohort.
The \emph{third idea} is to use the few available linear conditions 
to render some sets in a cohort benign. The benign sets pose no threat,
but still contribute elements towards the average set size in a cohort because
they contain the banner. (At this point the term ``average set size'' becomes a misnomer
since in the formula $\frac{\text{number of elements}}{\text{number of sets}}$
the benign sets contribute to the numerator, but not to the denominator. However we will keep
on using the term.)

To understand how linear conditions in a cohort work, we imagine that
sets in a cohort engage in an elimination tournament. When two nasty sets
$S$ and $S'$ are matched against one another instead of two invariants
$\chi(S)=C$ and $\chi(S')=C'$, there is just one 
invariant $\chi(S)+\chi(S')=C+C'$. The match is declared finished when the total
size of $S$ and $S'$ drops below $d$. The loser is the smaller of $S$ and $S'$, for it
can be shown that it became benign. The winner, on the other hand, might have
even larger error in $\chi(S)$ after the end of the match. 
Fortunately for us, the winner also gets a virtual trophy --- the banner of the loser.
This means that for purposes of computing the average set size
we can count the winner's banner element twice. In general,
a winner of $D$ matches will have $2^D$ virtual banner elements.

The final obstacle is the possibility that the banner of a cohort might get frozen. In that
case the defeated sets cease to contribute to the average set size, and 
the argument collapses. The rescue comes from the fact that when a
cohort was formed the banner was an element satisfying $\abs{\chi(b)}>1-\alpha$.
So, by replacing the invariant $\chi(S)=C$ by $\chi(S)+\beta \chi(b)=\hat{C}$,
where $\beta$ is very large, we can ensure that $b$ can be frozen only 
to the value that we originally attempted to round it to.  So, the banner can get frozen 
only in the favourable direction. If that happens, the sets in the cohort
become a bit less dangerous, but not yet benign; 
so we dissolve the cohort, and return the sets to the general pool, where
they can become parts of new cohorts. After being in a cohort $\Delta$
times, a set is guaranteed to be benign, and so the error in $\chi$ cannot
explode.\medskip

Despite its length, the sketch above misses a couple of crucial technical moments
that can be seen only from details. Most crucially, it does not explain why
the final bound, $\log^* d$, is pitifully tiny. It is
because the winners not only accumulate virtual banners, 
but also lose some of their own elements to freezing. 

It turns out that the banner accumulation happens faster, but the reason is subtle: Let $R_D$ be the maximal possible error in $\chi$
in a set that defeated $D$ sets. Crudely, we may think of $R_D$ as the number of
elements that a set lost due to freezing.
One can show that it satisfies an approximate recurrence $R_{D+1}\approx 2R_D-2^D$, where the 
term $2^D$ reflects $2^D$ virtual banners of a loser. Though $R_D$ eventually becomes negative, 
it grows exponentially at first. In particular, a cohort might get disbanded when $\chi$ is
exponential in $R_0$. For any single set, this might happen up to $\Delta$ times, and so the error might grow to be a 
tower of height~$\Delta$.

\paragraph{Acknowledgments.} The author owes the development some of the ideas in this paper to discussions
with Po-Shen Loh, to whom he is very grateful. The author is also very thankful to the referee whose careful reading of the paper helped to eliminate some hard-to-catch mistakes. 
All the remaining mistakes remain responsibility of the author.

\section{Two-coloring algorithm}
\paragraph{Data.}
The only data that the Beck--Fiala algorithm remembers is the floating coloring $\chi\colon X\to [-1,+1]$. 
The data that our algorithm uses is more elaborate:

\begin{center}
\begin{tabularx}{410pt}{XX}\toprule
\textbf{Data}                                      &  \textbf{Informal meaning}\\\hline
Function $\chi\colon X\to [-1,1]$                  &  Floating coloring\\[0.5ex]
\multirow{3}{180pt}{Partition of the family $\F$ as\\
  $\F=\B\cup\G\cup\C_1\cup\dotsb\cup \C_m$}         &  $\B$ contains benign sets\\
                                                   &  $\G$ is a general pool\\
                                                   &  $\C_1,\dotsc,\C_m$ are cohorts\\[0.5ex]
Elements $b_1,\dotsc,b_m\in X$                     & Banners for the cohorts\\
Signs $\eps_1,\dotsc,\eps_m\in\{-1,+1\}$           & The color of banner $b_i$ is close to $\eps_i$\\
Integers $r_1,\dotsc,r_m\in\{0,\dotsc,\Delta-1\}$  & The larger $r_i$ is, the less dangerous $\C_i$ is\\
Set families $M_1,\dotsc,M_m$                      & $M_i$ is a current matching in a cohort $\C_i$\\
Integer $D[S]\in \{0,1,\dotsc\}$ for each $S\in \C_i$   & Set $S$ has defeated $D[S]$ sets in its cohort\\
\bottomrule
\end{tabularx}
\end{center}

\paragraph{Notation.} To describe the algorithm, and the invariants that data satisfies, we need notation which we introduce now.
An element $x\in X$ such that $\chi(x)\in \{-1,+1\}$ is called \emph{frozen}; otherwise $x$ is \emph{floating}.
With this in mind, here is our notation:
\begin{align*}
  \Sz(S)&\eqdef \sum_{\makebox[7ex]{$\scriptstyle x\in S\text{ floating}$}} 1&&\quad\text{number of floating elements in $S$},\\
  \chi(S)&\eqdef \sum_{\makebox[7ex]{$\scriptstyle x \in S$}}\chi(x) &&\quad\text{current color (discrepancy) of $S$},\\
  \Fr(S)&\eqdef \sum_{\makebox[7ex]{$\scriptstyle x\in S\text{ frozen}$}}  \chi(x)&&\quad\text{total color of frozen elements in $S$},\\
  \Fl(S)&\eqdef \sum_{\makebox[7ex]{$\scriptstyle x\in S\text{ floating}$}} \chi(x)&&\quad\text{total color of floating elements in $S$},\\
  \Th-(S)&\eqdef\Sz(S)-\Fr(S)&&\quad\text{threat of negative discrepancy},\\
  \Th+(S)&\eqdef\Sz(S)+\Fr(S)&&\quad\text{threat of positive discrepancy},\\
  \Th(S)&\eqdef \max\bigl(\Th-(S),\Th+(S)\bigr)&&\quad\text{threat of discrepancy}.
\end{align*}

For later use we record two identities,
\begin{align} 
\Th-(S) &= \Sz(S)+\Fl(S)-\chi(S),\label{eq:Thn}\\
\Th+(S) &= \Sz(S)-\Fl(S)+\chi(S).\label{eq:Thp}
\end{align}

The algorithm proceeds in stages. To refer to the data at a particular stage we use
superscripts. For example, $\Sz^{(n)}(S)$ and $D^{(n)}[S]$ denote the values
of $\Sz(S)$ and $D[S]$ respectively at the $n$'th stage of the algorithm. The notation
for other quantities follows the same pattern.

\paragraph{Constants.} In the proof we use several constants, which we introduce now. 
Their informal meanings appear to the right of their definitions.
\begin{align}
\Delta&\eqdef \log^* d                      &&\quad\text{A set $S$ is benign if $\Th(S)\leq 2d-\Delta$},\notag\\
\alpha&\eqdef \tfrac{1}{4}                  &&\quad\text{The threshold for rounding is }1-\alpha,\notag\\
\Tw_r&\eqdef 
  \begin{cases}
    \Delta                     & \text{if }r\in \{0,1\}\\
    2^{8\Tw_{r-2}}             & \text{otherwise}
  \end{cases}                               &&\quad\text{Tower function controlling blowup of }\operatorname{Th},\notag\\ 
\beta_r&\eqdef 4 \Tw_r                      &&\quad\text{Clamping factor for cohorts with }r_i=r,\notag\\
W      &\eqdef d/(64\Delta^2 \Tw_{\Delta-1})  &&\quad\text{Size of the newly-created cohorts},\notag\\
R_D    &\eqdef (D-2)2^D-(2^D-1)\Delta+2     &&\quad\text{This term controls $\chi(S)$ in a cohort}.\label{eq:Rddef}
\end{align}
Of these constants, the most important is $R_D$. The whole proof is based on the
fact that $R_D$ is eventually larger than $C 2^D$ for any constant $C$. 

In what follows we assume that $d$ is so large that $\Delta>10$ holds.

The choice of constants is fairly flexible, and the chosen constants are far from unique or optimal.
For example, any choice for $\alpha$ from $\Tw_{\Delta}/d$ to a constant less than $1$ would have worked (with minor changes
in the other constants). 
To avoid burdening the main exposition, we record the 
needed inequalities between constants here (valid for all $r<\Delta$):
\begin{equation}\label{eq:const_inequalities}
\begin{aligned}
\text{a) }&\Delta\leq \Tw_r&\qquad \text{b) }&\Tw_r+\Delta<\beta_r(1-\alpha)\\
\text{c) }&(\beta_r+1)\alpha+\Tw_r+\Delta\leq 4\Tw_r&\qquad\text{d) }&16\Tw_r\leq 2^{16\Tw_{r-2}}&\\
\text{e) }&\Tw_{\Delta-1}\leq \log \log d&\qquad\text{f) }&2\Tw_{\Delta-1}+(2-\alpha)\Delta\leq \alpha d\\
\text{g) }&d\geq W \Delta \bigl(\Delta+2\Tw_{\Delta-1}+2\bigr)(8+4\Delta)&\qquad\text{h) }&2^{8\Tw_r+6}\Tw_r<W\\[0.2ex]
\text{i) }
\mathrlap{(1-\tfrac{1}{2}\alpha)(2d-\Delta)\geq (1-\tfrac{1}{2}\alpha)d+2^{4\Tw_r+1}\beta_r\alpha+2^{4\Tw_r+1}\Tw_r+2^{4\Tw_r}\Delta}\\
\text{j) }&\mathrlap{\Tw_{r+2}\geq 2^{4\Tw_r+1}\beta_r\alpha+2^{4\Tw_r+1}\Tw_r+2^{4\Tw_r}\Delta}\\
\end{aligned}
\end{equation}
\textit{Proofs:}  
\textit{(a)} is proved by induction on $r$;\quad 
\textit{(b)} and \textit{(c)} are implied by (a);\quad 
\textit{(d)} follows from $1\leq \Delta$ and (a);\quad 
\textit{(e)} is true because $\log^*(16\Tw_r)\leq r/2+\log^*(16\Delta)$ by 
  induction on $r$, and because
  $\Delta/2+\log^*(16\Delta)\leq \Delta-1$ for $\Delta> 10$;\quad 
\textit{(f)} follows from (e);\quad 
\textit{(g)} follows from (e) and the definition of W;\quad
\textit{(h)} is implied by (a) and (e);\quad 
\textit{(i)} follows from (a) and three uses of (e) to bound each of the summands on the right;\quad
\textit{(j)} is a consequence of (a) and the definition of $\Tw_r$.\qed

\paragraph{Invariants.} The data satisfies the following fourteen invariants.
Note that the invariants are symmetric with respect to flipping the coloring,
i.e., with respect to the inversion $\chi\mapsto -\chi$, $\eps_i\mapsto -\eps_i$. 
Our algorithm is also symmetric in this sense. We suggest that the reader 
examine \cref{i:frozen,i:benign,i:dangerbound,i:banner,i:matching,i:eqmatching,i:smallsize},
and refer to other invariants later.

\begin{invlist}
\item \label{i:frozen}
   If $x\in X$ is frozen in coloring $\chi^{(i)}$, then it is also frozen in $\chi^{(j)}$ for all $j>i$.
\item \label{i:benign}
   Every set $S\in \B$ satisfies $\Th(S)\leq 2d-\Delta$.
\item \label{i:dangerbound}
   Every set $S\in \F$ satisfying $\Sz(S)\leq d$ also satisfies $\Th(S)\leq 2d$.
\item \label{i:banner}
   Element $b_i$ is common to all the sets in $\C_i$.
\item \label{i:matching}
   Family $M_i$ is a matching on $\C_i$, i.e., $M_i$ consists of pairs
   of sets from $\C_i$, and the pairs in $M_i$ are disjoint. 
   Note that the matching is not perfect: some sets in $\C_i$ might be unmatched.
\item \label{i:eqmatching}
   If $\{S,S'\}\in M_i$, then $D[S]=D[S']$.
\item \label{i:smallsize}
   If $S\in\C_i$, then $\Sz(S)\leq d+1-2^{D[S]}$.
\item \label{i:otherthreatp}
   If $S\in\C_i$, and $\eps_i=+1$ then 
      $\Th+(S)\leq \Delta+2-2^{D[S]+1}$.
\item \label{i:otherthreatn}
   If $S\in\C_i$, and $\eps_i=-1$ then 
      $\Th-(S)\leq \Delta+2-2^{D[S]+1}$.
\item \label{i:bannercnt}
   If element $b\in X$ is a banner of $k$ cohorts $\C_{i_1},\dotsc,\C_{i_k}$, then
      $b$ is contained in at least $\sum_{j=1}^k \bigl(W-\abs{\C_{i_j}}\bigr)$
      sets of $\B$.
\item \label{i:elim} 
     For each cohort $\C_i$ we have $\sum_{S\in \C_i} 2^{D[S]}\leq W$.
     \addtolength{\labelwidth}{1.4ex}
     \addtolength{\itemindent}{1.4ex}
\item \label{i:Gchip}
     If $S\in \G$ and $\Th+(S)=2d-r$, then\\[1ex]
     $\displaystyle\begin{aligned}
\phantom{-}\chi(S)&\leq 0&&\quad\text{if }r<\Delta\text{ and }\Sz(S)\geq d,\\
\phantom{-}\chi(S)&\leq \Tw_r-\tfrac{1}{2}\alpha \Th-(S)&&\quad\text{if }r<\Delta\text{ and }\Sz(S)\leq d.\\
     \end{aligned}$
\item \label{i:Gchin}
     If $S\in \G$ and $\Th-(S)=2d-r$, then\\[1ex]
     $\displaystyle\begin{aligned}
      -\chi(S)&\leq 0&&\quad\text{if }r<\Delta\text{ and }\Sz(S)\geq d,\\
      -\chi(S)&\leq \Tw_r-\tfrac{1}{2}\alpha \Th+(S)&&\quad\text{if }r<\Delta\text{ and }\Sz(S)\leq d.\\
     \end{aligned}$
\item \label{i:Cchifreen}
     If $\eps_i=-1$, and a set $S\in \C_i$ is not in any edge of $M_i$, and $D=D[S]$ then\\[1ex]
     $\displaystyle
     \phantom{-}\chi(S)+2^D \beta_{r_i}\bigl(\chi(b_i)+1-\alpha\bigr) 
         \leq 2^D \Tw_{r_i} - \tfrac{1}{2}\alpha \Th-(S) - R_D.
     $
\item \label{i:Cchifreep}
     If $\eps_i=+1$, and a set $S\in \C_i$ is not in any edge of $M_i$, and $D=D[S]$ then\\[1ex]
     $\displaystyle
     -\chi(S)-2^D \beta_{r_i}\bigl(\chi(b_i)-1+\alpha\bigr) 
         \leq 2^D \Tw_{r_i} - \tfrac{1}{2}\alpha \Th+(S) - R_D.
     $
\item \label{i:Cchimatchedn}
     If $\eps_i=-1$, and $\{S,S'\}\in M_i$, and $D=D[S]=D[S']$ then\\[1ex]
     $\displaystyle
     \phantom{-}\chi(S)+\chi(S')+2^{D+1} \beta_{r_i}\bigl(\chi(b_i)+1-\alpha\bigr) 
         \leq 2^{D+1} \Tw_{r_i} - \tfrac{1}{2}\alpha\bigl(\Th-(S)+\Th-(S')\bigr) - 2 R_D.
     $
\item \label{i:Cchimatchedp}
     If $\eps_i=+1$, and $\{S,S'\}\in M_i$, and $D=D[S]=D[S']$ then\\[1ex]
     $\displaystyle
     -\chi(S)-\chi(S')-2^{D+1} \beta_{r_i}\bigl(\chi(b_i)-1+\alpha\bigr) 
         \leq 2^{D+1} \Tw_{r_i} - \tfrac{1}{2}\alpha\bigl(\Th+(S)+\Th+(S')\bigr) - 2 R_D.
     $
\item \label{i:Dbound}
     If $S\in \C_i$, then $D[S]\leq 4\Tw_{r_i}$.
\end{invlist}

\paragraph{Initialization of the algorithm.} At start,
we set $\chi=0$, $\B=\emptyset$, $\G=\F$, $\C=\emptyset$ and $m=0$.

\paragraph{Steps of the algorithm.} The algorithm of Beck--Fiala
makes progress in two ways: by using linear perturbation of the current
floating coloring, and by discarding the benign sets. Our algorithm
uses nine ways to make progress. We refer to these ways as \emph{steps}.
For each step, there is a condition that must hold for it to be executed.
The steps are ordered, and the choice of a step to be executed is greedy: 
we always execute the first step whose condition is satisfied.
For example, step 5 is executed only if steps 1 through 4 cannot be executed.

{%
\step{Moving benign sets from $\G$ to $\B$}{There is an $S\in \G$ such that $\Th(S)\leq 2d-\Delta$}{s:genbenign}
  Move $S$ from $\G$ to $\B$. 

\step{Removing empty cohorts}{There is a cohort $\C_i$ such that $\C_i=\emptyset$}{s:emptycohort}
  Remove cohort $\C_i$ by deleting $\C_i$, $b_i$, $\eps_i$, $r_i$ and $M_i$
  and renumbering the remaining cohorts appropriately.

\stepalt{Rounding elements that are safe to round}%
{There is an $x\in X$ such that $\chi(x)>1-\alpha$ and there exists no $S\in\G$ containing $x$ that satisfies $\Sz(S)\geq d+1$ and $\Th+(S)>2d-\Delta$}%
{There is an $x\in X$ such that $-\chi(x)>1-\alpha$ and there exists no $S\in\G$ containing $x$ that satisfies $\Sz(S)\geq d+1$ and $\Th-(S)>2d-\Delta$}
\label{s:round}
  In case ($+$), set $\chi(x)$ equal to $+1$.
  In case ($-$), set $\chi(x)$ equal to $-1$.

\step{Moving benign cohort sets to $\B$}{There is an $S\in \C_i$ that is not a part of $M_i$
and such that $\Th(S)\leq 2d-\Delta$}{s:cohbenign}
  Move $S$ to $\B$.

\step{Disbanding cohorts whose banners were frozen}{For some $i$ we have $\chi(b_i)\in \{-1,+1\}$}{s:frozencohort}
  Disband cohort $\C_i$ by moving all sets in $\C_i$ to $\G$, and then 
  removing cohort as in \cref{s:emptycohort}.

\step{Declaring some matches finished}{There is an edge $\{S,S'\}\in M_i$ with $D=D[S]=D[S']$ such that
\begin{equation}\label{eq:finishcond}
\Sz(S)+\Sz(S')+2^{D+1}-2\leq d.
\end{equation}\steppointfalse}{s:exhaust}
  Without loss of generality, $\Sz(S)\geq \Sz(S')$. Perform the next three actions:\\*
  1) Remove edge $\{S,S'\}$ from $M_i$.\\* 
  2) Move $S'$ to $\B$.\\*
  3) If $\Th(S)\leq 2d-\Delta$, move $S$ to $\B$ as well. Otherwise, increment $D[S]$.

\step{Matching unmatched sets in a cohort}{There are sets $S,S'\in \C_i$ 
that are not part of $M_i$ and such that $D[S]=D[S']$}{s:match}
  Add edge $\{S,S'\}$ to $M_i$.

\step{Linear perturbation}{Let $N=\abs{\G}+\sum_{i=1}^m \bigl(\abs{\C_i}-\abs{M_i}\bigr)$.
Execute this step only if the number of floating elements exceeds $N$}{s:perturb}
  We will generate a set $E$ of $N$ linear equations in values of 
  an unknown function $\tau\colon X\to \R$.
  The equations will have the property that the 
  current floating coloring $\chi$ satisfies them all.

  - Each $S\in \G$ generates the equation $\tau(S)=\chi(S)$.\\*
  - Each $S\in \C_i$ that is not in $M_i$ generates the equation
  \[
    \tau(S)+2^D \beta_{r_i}\tau(b_i)=\chi(S)+2^D \beta_{r_i}\chi(b_i).
  \]
  - Each edge $\{S,S'\}\in M_i$ generates the equation
  \[
    \tau(S)+\tau(S')+2^{D+1} \beta_{r_i}\tau(b_i)=\chi(S)+\chi(S')+2^{D+1} \beta_{r_i}\chi(b_i).
  \]

  Let $E$ be the resulting set of equations. Note that $\abs{E}=N$. Let
  \[
    A\eqdef \bigl\{ \tau : \tau\text{ satisfies all of }E,\text{ and }\tau(x)=\chi(x)\text{ whenever }\chi(x)\in\{-1,+1\}\,\bigr\}
  \]
  Since the number of floating elements exceeds $N$, set $A$ is an affine space of positive dimension.
  Since $A$ contains $\chi$, it also must contain a point $\tau\in [-1,+1]^X$ in which more elements
  are frozen than in the current value of $\chi$. Set $\chi$ to that $\tau$.

\step{Creating a new cohort}{Not all sets are in $\B$, i.e., $\F\neq \B$}{s:newcohort}
  In \cref{sec:newcohort} we will show that this step is executed only if
  there is a $b\in X$, and a family $\Cb\subset \G$ of size $\abs{\Cb}=W$ 
  and a number $r<\Delta$ such that one of the following holds.\\*
\makebox[3.4ex][r]{$+$) }We have $\chi(b)>1-\alpha$, and each $S\in\Cb$ satisfies $b\in S$, $\Th-(S)=2d-r$ and $\Sz(S)\leq d$.\\*
\makebox[3.4ex][r]{$-$) }We have $-\chi(b)>1-\alpha$, and each $S\in\Cb$ satisfies $b\in S$, $\Th+(S)=2d-r$ and $\Sz(S)\leq d$.

Create a new empty cohort $\C_{m+1}$, move all sets in $\Cb$ from $\G$
to $\C_{m+1}$, and set $b_{m+1}=b$, $\eps_{m+1}=\sign \chi(b)$, $r_{m+1}=r$, $M_{m+1}=\emptyset$
and $D[S]=0$ for all $S\in \C_{m+1}$.

}

\section{Proof of the algorithm's correctness}
In this section we show that each step of the algorithm preserves all the invariants
enumerated in the previous section. We also show that the algorithm terminates, and
that its termination implies that the discrepancy of the set family $\F$ is at most 
$2d-\Delta$.

In the proofs that follow we use several consequences of the invariants that we stated.
We record these consequences now.

\begin{lemma}\label{lem:thmonotone}
If \cref{i:frozen} holds, then $\Th+^{(j)}(S)$ and $\Th-^{(j)}(S)$ are non-increasing functions of $j$.
\end{lemma}
\begin{proof}
Suppose $j<k$. Let $r=\Sz^{(j)}(S)-\Sz^{(k)}(S)$ be the number of elements
frozen between stages $j$ and $k$. Then  $\Fr^{(k)}\leq \Fr^{(j)}+r$. So 
$\Th+^{(k)}(S)=\Sz^{(k)}(S)+\Fr^{(k)}(S)\leq \bigl(\Sz^{(j)}(S)-r\bigr)+\bigl(\Fr^{(j)}(S)+r\bigr)$.
The proof for $\Th-$ is similar.
\end{proof}
\begin{lemma}\label{lem:genonesided_benign} 
If $S\in \F$ satisfying $\Sz(S)\leq d$ also satisfies $\Th+(S)>2d-\Delta$, then $\Th-(S)<\Delta$.
Conversely, if $\Th-(S)>2d-\Delta$, then $\Th+(S)<\Delta$.
\end{lemma}
\begin{proof}
This follows from $\Th+(S)+\Th-(S)=2\Sz(S)$.
\end{proof}
\begin{lemma}\label{lem:onesided_benign}
If \cref{i:frozen} holds, then
  $\Th+(S)\leq 2d-\Delta$ whenever $S\in \C_i$ and $\eps_i=+1$. 
Similarly, if \cref{i:frozen} holds, then 
  $\Th-(S)\leq 2d-\Delta$ whenever $S\in \C_i$ and $\eps_i=-1$. 
\end{lemma}
\begin{proof} By symmetry, we may assume $\eps_i=+1$.
In view of \cref{lem:thmonotone}, it suffices to consider only
the stage when $S$ is added to a cohort at \cref{s:newcohort}.
At that stage, \cref{lem:genonesided_benign} applies. 
\end{proof}
\begin{lemma}\label{lem:other_chi_bound}
Suppose \cref{i:otherthreatp,i:otherthreatn} hold. If $S\in \C_i$, then 
\[
 \eps_i\cdot \chi(S)-\gamma\Th{\eps_i}(S)\leq \Delta+2-2^{D[S]+1}
\]
for all $0\leq\gamma\leq 1$.
\end{lemma}
\begin{proof}
Assume $\eps_i=+1$ by symmetry. 
We have
\begin{align*}
  \chi(S)-\gamma\Th+(S)&=(1-\gamma)\chi(S)-\gamma\bigl(\Sz(S)-\Fl(S)\bigr)&&
     \text{using identity \eqref{eq:Thp}}\\
   &\leq (1-\gamma)\Th+(S)&&
      \text{due to }\chi\leq \Th+\text{ and }\Fl\leq\Sz\\
   &\leq (1-\gamma)(\Delta+2-2^{D[S]+1})&&
      \text{because of \cref{i:otherthreatp}}.\qedhere
\end{align*}
\end{proof}
\begin{lemma}\label{lem:bannersign}
If all the invariants hold and cohort $\C_i$ is non-empty, then
$\sign \chi(b_i)=\eps_i$.
\end{lemma}
\begin{proof}
Consider the case $\eps_i=+1$, the other case is analogous.
If $M_i$ is non-empty, let $\{S,S'\}$ be any edge in $M_i$. If $M_i$ is empty,
let $S$ be any set in $\C_i$, and put $S'=S$. In either case, \cref{i:Cchifreep,i:Cchimatchedp} 
imply that
\[
  -\chi(S)-\chi(S')-2^{D+1} \beta_{r_i}\bigl(\chi(b_i)-1+\alpha\bigr) 
          \leq 2^{D+1} \Tw_{r_i} - \tfrac{1}{2}\alpha\bigl(\Th+(S)+\Th+(S')\bigr) - 2 R_D.
\]
By the preceding lemma with $\gamma=\tfrac{1}{2}\alpha$ and a bit of algebra, where
nearly all terms cancel, it follows that
\begin{align*}
 -\beta_{r_i}\bigl(\chi(b_i)-1+\alpha\bigr)&\leq \Tw_{r_i}-D+\Delta,
\end{align*}
and the lemma follows since $\Tw_{r_i}-D+\Delta<\beta_{r_i}(1-\alpha)$
by \refconst{b}.
\end{proof}
\begin{lemma}\label{lem:Dbound}
If \cref{i:smallsize,i:Cchifreen,i:Cchifreep} hold, then for every $S\in \C_i$
that is not in $M_i$ and satisfying $\Th(S)>2d-\Delta$, \cref{i:Dbound} holds.
\end{lemma}
\begin{proof}
Say $\eps_i=+1$, the other case being symmetric. For brevity write $D=D[S]$ and $r=r_i$.
\Cref{lem:onesided_benign} asserts that
$\Th+(S)\leq 2d-\Delta$, and so $\Th(S)>2d-\Delta$ implies that 
\begin{equation}\label{eq:lembndn}\Th-(S)>2d-\Delta.\end{equation} 
We thus have
\begin{align*}
  0&\leq 2^D \beta_{r}\alpha + 2^D \Tw_r - R_D +\chi(S) - \tfrac{1}{2}\alpha \Th+(S) &&\text{from \cref{i:Cchifreep}}\\
   &=2^D \beta_{r}\alpha + 2^D \Tw_r - R_D + (1-\alpha)\Sz(S)+\Fl(S)-(1-\tfrac{1}{2}\alpha)\Th-(S)
        &&\text{using identities \eqref{eq:Thn} and \eqref{eq:Thp}}\\
   &\leq 2^D \beta_{r}\alpha + 2^D \Tw_r - R_D + (2-\alpha)\Sz(S)-(1-\tfrac{1}{2}\alpha)(2d-\Delta)&&\text{from $\Fl(S)\leq \Sz(S)$ and \eqref{eq:lembndn}}\\
   &\leq 2^D \beta_{r}\alpha + 2^D \Tw_r - R_D + (2-\alpha)(d+1-2^D)-(1-\tfrac{1}{2}\alpha)(2d-\Delta)&&\text{from \cref{i:smallsize}}\\
   &= 2^D \bigl(\beta_{r}\alpha + \Tw_r + \Delta + \alpha -D \bigr) -\tfrac{1}{2}\alpha \Delta-\alpha &&\text{from \eqref{eq:Rddef} and some algebra}
\end{align*}
and so $D\leq 4\Tw_r$ by \refconst{c}. Hence \cref{i:Dbound} holds.
\end{proof}

It is easy to check that the algorithm satisfies all the invariants at the initialization stage.
Only \cref{i:Gchip,i:Gchin} require an invocation of \cref{lem:genonesided_benign}
and of \refconst{a}; the other invariants are immediate.

In what follows we assume that the algorithm satisfies 
all the invariants at stage $n-1$, and that our goal is to show
that the algorithm satisfies them at stage $n$. For brevity,
we write $\Sz(S), \Fr(S)$ etc in place of $S^{(n-1)}(S)$, 
$\Fr^{(n-1)}(S)$ etc. We still write $\Sz^{(n)}(S)$ etc in full.

\verif{i:frozen}
The only steps that modify $\chi$ are
\cref{s:round,s:perturb}. They do not unfreeze any elements.

\verif{i:benign}
The only steps that move sets into $\B$
are \cref{s:genbenign,s:cohbenign,s:exhaust}. The moves in 
\cref{s:genbenign,s:cohbenign} are preconditioned on $\Th(S)\leq 2d-\Delta$,
and so \cref{i:benign} holds trivially. The only potential problem
is the movement of $S'$ to $\B$ in \cref{s:exhaust}, which we now tackle.

Assume that $\eps_i=+1$, the other case is symmetric. 
Since $\Sz(S')\leq \Sz(S)$, it follows from the execution 
condition of \cref{s:exhaust} that 
$\Sz(S')\leq d/2$. By \cref{lem:onesided_benign}, $\Th+(S')\leq 2d-\Delta$, so
it remains to prove that $\Th-(S')\leq 2d-\Delta$.

Use of the identity \eqref{eq:Thn}, inequality $\Fl(S')\leq \Sz(S')$,
and  \cref{lem:other_chi_bound} in that order, yields
\begin{align*}
\Th-(S')
        &\leq 2\Sz(S')-\chi(S')
        \leq 2\Sz(S')-\chi(S')-\chi(S)+\tfrac{1}{2}\alpha\Th+(S)+\Delta+1-2^D.\\
\intertext{Applying \cref{i:Cchimatchedp} and using $\Th+(S')+\Th-(S')=2\Sz(S')$ gives}
\Th-(S')&\leq 2\Sz(S')-\tfrac{1}{2}\alpha \Th+(S')+2^{D+1}\beta_{r_i} \alpha+2^{D+1}\Tw_{r_i}-2R_D+\Delta\\
        &=(2-\alpha)\Sz(S')+\tfrac{1}{2}\alpha \Th-(S') +2^{D+1}\beta_{r_i} \alpha+2^{D+1}\Tw_{r_i}-2R_D+\Delta.
\end{align*}
The \cref{i:benign} then follows by 
using $\Sz(S')\leq d/2$, \cref{i:Dbound} and \refconst{i}.

\verif{i:dangerbound} Since $\Th(S)$ is non-increasing by \cref{lem:thmonotone}, 
only the steps that might decrease $\Sz(S)$ need to be examined. These are \cref{s:round,s:perturb}.
So, let us consider a set $S$ such that $\Sz(S)>d$ and $\Sz^{(n)}(S)\leq d$.

We will prove that $\Th+^{(n)}(S)\leq 2d$ (the other case is similar). We may assume that 
$\Th+(S)>2d$, for else we are done by \cref{lem:thmonotone}. By \cref{i:Gchip}, $\chi(S)\leq 0$.
Furthermore, for \cref{s:perturb} we have $\chi^{(n)}(S)=\chi(S)$, whereas
for \cref{s:round} we have $\chi^{(n)}(S)\leq \chi(S)+\alpha$. Hence,
identity \eqref{eq:Thp} implies
\[
  \Th+^{(n)}(S)=\Sz^{(n)}(S)-\Fl^{(n)}(S)+\chi^{(n)}(S)\leq 2\Sz^{(n)}(S)+\alpha\leq 2d+\alpha.
\]
Since $\Th+^{(n)}(S)$ is an integer, we conclude that $\Th+^{(n)}(S)\leq 2d$ as desired.

\verif{i:banner} The only step that adds sets to a cohort, or modifies a
cohort's banner is \cref{s:newcohort}. However, that step clearly respects
\cref{i:banner}.

\verif{i:matching} The only step that adds edges to $M_i$ is
\cref{s:match}. It adds edges only between unmatched sets.

\verif{i:eqmatching} The only step that adds edges to $M_i$ is
\cref{s:match}. It adds edges only between sets with equal value of $D$.
The only steps that change value of $D[S]$ are \cref{s:exhaust,s:newcohort}.
\Cref{s:exhaust} changes $D[S]$ only after having removed the edge
that contains $S$ from $M_i$. \Cref{s:newcohort} is not a problem either, as
the cohort that it creates has empty matching.

\verif{i:smallsize} In view of \cref{i:frozen}, $\Sz(S)$ can only decrease.
Thus, it suffices to verify only the steps that either 
increase the value of $D[S]$ or add sets to $\C_i$. 
These are \cref{s:exhaust,s:newcohort} respectively.\medskip

\Cref{s:newcohort} adds sets with $\Sz(S)\leq d$ and assigns $D[S]=0$ for them,
satisfying \cref{i:smallsize}.\medskip

Consider \cref{s:exhaust}. Since it is that step which is executed, 
and not \cref{s:frozencohort}, it follows that $b_i$ is not frozen,
and so $\Sz(S')\geq 1$. Hence, 
from the inequality \eqref{eq:finishcond} we have
\[
  \Sz(S)\leq d+1-2^{D[S]+1}=d+1-2^{D^{(n)}[S]}. 
\]

\verif{i:otherthreatp,i:otherthreatn} 
By \cref{lem:thmonotone} it suffices 
to check only the steps that move sets to $\C_i$ or increase the value of $D[S]$. These
are \cref{s:exhaust,s:newcohort}. 

By symmetry it suffices to treat only \cref{i:otherthreatp}. 
Note that both in \cref{s:exhaust} and in \cref{s:newcohort}, the sets $S$ 
for which we need to establish \cref{i:otherthreatp}
 satisfy $\Th-^{(n)}(S)>2d-\Delta$. Hence,
\[\Th+^{(n)}(S)=\Sz^{(n)}(S)+\Fr^{(n)}(S)=2\Sz^{(n)}(S)-\Th-^{(n)}(S)<2\Sz^{(n)}(S)-2d+\Delta,\]
and \cref{i:otherthreatp} follows from the validity of \cref{i:smallsize} at stage $n$,
which was proved immediately above.

\verif{i:bannercnt} We need to check only the steps that remove
sets from $\C_i$, as well as creation of new cohorts in \cref{s:newcohort}.
The latter is trivial since  $\abs{\Cb}=W$. The only steps that remove
sets from $\C_i$ are \cref{s:cohbenign,s:frozencohort,,s:exhaust}. Of these,
\cref{s:cohbenign,s:exhaust} move sets from $\C_i$ to $\B$, thus preserving
\cref{i:bannercnt}. \Cref{s:frozencohort} removes sets to $\G$, but
also disbands the cohort. The fact that the step preserves \cref{i:bannercnt}
follows from $\abs{\C_i}\leq W$, which is a consequence of \cref{i:elim}.

\verif{i:elim} We need to check only the steps that either 
increase the value of $D[S]$ or add sets to $\C_i$. 
These are \cref{s:exhaust,s:newcohort}. For \cref{s:exhaust}
we note that $D[S]=D[S']$ by \cref{i:eqmatching},
and so $2^{D^{(n)}(S)}=2^{D[S]+1}=2^{D[S]}+2^{D[S']}$. For
\cref{s:newcohort} the invariant follows from $\abs{\Cb}=W$
and $D[S]=0$.

\verif{i:Gchip,i:Gchin} We only treat \cref{i:Gchip}, for \cref{i:Gchin} is symmetric.
We need to check only the steps that modify $\chi$ or move sets to $\G$. These 
are \cref{s:round,s:perturb,s:frozencohort}. Let $S\in \F$ be arbitrary, 
and let us check that the invariant holds for $S$.\medskip

\Cref{s:perturb} does not change $\chi(S)$, and may only decrease
$\Th+(S)$ and $\Th-(S)$, by \cref{lem:thmonotone}. So, if 
$\Sz(S),\Sz^{(n)}(S)$ are either both greater than $d$, or
are both at most $d$, then the invariant holds at stage $n$
because it held at stage $n-1$. So, consider the
case when $\Sz(S)>d$ and $\Sz^{(n)}(S)\leq d$.
Note that since $\Sz(S)>d$ and \cref{i:Gchip}
held at stage $n-1$, it follows that $\chi(S)\leq 0$.
If $\Th+^{(n)}(S)\leq 2d-\Delta$ then \cref{i:Gchip} holds
vacuously. Otherwise, 
$\Th-^{(n)}(S)<\Delta$ by \cref{lem:genonesided_benign}.
Thus
$\Tw_r-\tfrac{1}{2}\alpha \Th-^{(n)}(S)\geq 0$ by \refconst{a}.
Since $\chi^{(n)}(S)=\chi(S)\leq 0$, \cref{i:Gchip} holds.\medskip

\Cref{s:round} alters $\chi(S)$, $\Sz(S)$, $\Th-(S)$ or $\Th+(S)$ 
only if it rounds an element $x$ that is in $S$. So, assume $x\in S$.
There are two cases according to the sign of $\chi(x)$:
\begin{enumerate}
\item[$+$)] If $\chi(x)>1-\alpha$, then for the condition of \cref{s:round}
to have triggered, we must have either $\Th+(S)\leq 2d-\Delta$
or $\Sz(S)\leq d$. In the former case, $\Th+^{(n)}(S)\leq \Th+(S)$ by
\cref{lem:thmonotone}, and so \cref{i:Gchip} holds vacuously. In the latter
case, $\chi^{(n)}(S)\leq\chi(S)+\alpha$, $\Th+^{(n)}(S)=\Th+(S)$
and $\Th-^{(n)}(S)=\Th-(S)-2$, and the invariant is verified by 
substitution.
\item[$-$)] If $-\chi(x)>1-\alpha$, then $\chi^{(n)}(S)<\chi(S)$,
and $\Th+^{(n)}(S)=\Th+(S)-2$ and $\Th-^{(n)}(S)=\Th-(S)$.
As in the treatment of \cref{s:perturb} above, the only case worthy of attention
is $\Sz(S)=d+1$, $\Sz^{(n)}(S)=d$, and the same argument as above disposes of it.
\end{enumerate}\medskip

We treat \cref{s:frozencohort} next. Suppose cohort $\C_i$ is being dissolved, and
$S\in \C_i$ is an arbitrary set in it. If $\eps_i=+1$, then by \cref{lem:onesided_benign}
\cref{i:Gchip} holds vacuously. So, assume $\eps_i=-1$. If $S$ is not a part of $M_i$,
declare $S'=S$, otherwise let $\{S,S'\}\in M_i$ be the edge containing $S$. In both
cases we conclude (either from \cref{i:Cchifreen} or \cref{i:Cchimatchedn}) that 
\[
  \chi(S)+\chi(S')+2^{D+1} \beta_{r_i}\bigl(\chi(b_i)+1-\alpha\bigr) 
         \leq 2^{D+1} \Tw_{r_i} - \tfrac{1}{2}\alpha\bigl(\Th-(S)+\Th-(S')\bigr) - 2 R_D.
\]
An application of \cref{lem:other_chi_bound} to $-\chi(S')-\tfrac{1}{2}\alpha\Th-(S')$
gives
\begin{equation}\label{eq:frozenchibound}
\begin{aligned}
  \chi(S)&\leq  2^{D+1}\beta_{r_i}\alpha+ 2^{D+1} \Tw_{r_i} +\Delta+2-2^{D+1} - 2 R_D -\tfrac{1}{2}\alpha\Th-(S)\\
         &\leq \Tw_{r_i+2}-\tfrac{1}{2}\alpha \Th-(S),
\end{aligned}
\end{equation}
with the last line holding because of \refconst{j} and \cref{i:Dbound}. 

Let $j$ be the stage when cohort $\C_i$ was created. By the description
of \cref{s:newcohort} we had $\Th+^{(j)}(S)=2d-r_i$ and banner $b_i$ was a floating element. 
So, since $\chi(b_i)=-1$ by \cref{lem:bannersign}, we conclude that
$\Th+(S)\leq 2d-r_i-2$. In view of inequality \eqref{eq:frozenchibound} above, this implies that \cref{i:Gchip} holds.

\verif{i:Cchifreen,i:Cchifreep} We treat \cref{i:Cchifreen}, for \cref{i:Cchifreep}
is symmetric. We need to verify the steps that remove sets 
from $M_i$, modify $D[S]$ or $\chi$, or add sets to $\C_i$. These are 
\cref{s:round,s:exhaust,s:perturb,,s:newcohort}. We handle them in (reverse) order.\medskip

Validity of \cref{i:Cchifreen} after \cref{s:newcohort} follows from 
$\chi(b_i)+1-\alpha\leq 0$, $R_0=0$, $\Sz(S)\leq d$ and the validity of
\cref{i:Gchip} before the step.\medskip

\Cref{s:perturb} does not change the left-hand side of the inequality
in \cref{i:Cchifreen}, and might only increase the right-hand side (by decreasing $\Th-(S)$).\medskip

We treat \cref{s:exhaust} next. We need to consider only the case $\Th(S)>2d-\Delta$, 
for otherwise the invariant holds vacuously. Let $D=D[S]$. Our goal is to bound
\begin{align*}
  Q\eqdef \chi^{(n)}(S)+2^{D^{(n)}[S]} \beta_{r_i}\bigl(\chi(b_i)+1-\alpha\bigr)&=
  \chi(S)+2^{D+1} \beta_{r_i}\bigl(\chi(b_i)+1-\alpha\bigr).
\end{align*}
Since \cref{i:Cchimatchedn} held at stage $n-1$, we conclude that 
\begin{align*}
 Q&\leq 2^{D+1} \Tw_{r_i} - \tfrac{1}{2}\alpha\bigl(\Th-(S)+\Th-(S')\bigr)-\chi(S') - 2 R_D
    &&\text{by \cref{i:Cchimatchedn}}.
\end{align*}
Hence, it suffices to prove that $-\tfrac{1}{2}\alpha\Th-(S')-\chi(S')\leq 2R_D-R_{D+1}$. 
This follows from \cref{lem:other_chi_bound} and
from $2R_D-R_{D+1}=\Delta+2-2^{D+1}$. We are thus done with \cref{s:exhaust}.\medskip

Finally, we check \cref{s:round}. We may assume that
$x$, the element that is rounded, is in $S$. If $\chi(x)<1-\alpha$, then
rounding decreases left-hand side of the inequality in \cref{i:Cchifreen},
and does not affect the right-hand side at all. So, we may assume that
$\chi(x)>1-\alpha$. By \cref{lem:bannersign}
$\chi(b_i)<0$, and so $x\neq b_i$. The rounding thus increases the 
left-hand side of the inequality in \cref{i:Cchifreen} by at most $\alpha$.
Since the right-hand side of the inequality increases by exactly $\alpha$,
we are done.

\verif{i:Cchimatchedn,i:Cchimatchedp} We need to check only
\cref{s:round,s:match,s:perturb} as these are the only steps
that either change $\chi$ or create an edge in $M_i$. The verification
of \cref{s:round,s:perturb} is an almost verbatim repetition of the verification
of \cref{i:Cchifreen,i:Cchifreep} for those steps, and we omit it.
The validity of \cref{i:Cchimatchedn,i:Cchimatchedp} after \cref{s:match}
follows from the validity of \cref{i:Cchifreen,i:Cchifreep} respectively before the step.

\verif{i:Dbound} We need to consider only \cref{s:exhaust}
as it is the only step that increases the value of $D[S]$.
Since \cref{i:smallsize,i:Cchifreen,i:Cchifreep} have been proved above,
the \cref{i:Dbound} follows from \cref{lem:Dbound}.

\paragraph{Proof that the algorithm terminates:}
Let $F$ be the number of frozen elements, and consider the quantity 
\[
  I\eqdef F+ 4\abs{\B}-m+\sum_{i=1}^m \bigl(\abs{M_i}+\abs{\C_i}\bigr).
\] 
We claim that each step other than \cref{s:frozencohort} increases~$I$.
\Cref{s:genbenign} increases $\abs{\B}$; \cref{s:emptycohort} decreases $m$; \cref{s:round} increases $F$;
\cref{s:cohbenign} increases $4\abs{\B}$ by $4$ and decreases $\abs{\C_i}$ by $1$; \cref{s:exhaust} increases $4\abs{\B}$ by 
at least $4$, decreases $\abs{M_i}$ by $1$, and decreases $\abs{\C_i}$ by at most $2$; \cref{s:match} increases $\abs{M_i}$;
\cref{s:perturb} increases $F$; \cref{s:newcohort} increases $\sum_i\abs{\C_i}$ by $W$ and increases $m$ by $1$.  

\Cref{s:frozencohort} is executed only if there are no empty cohorts (\cref{s:emptycohort} cannot be executed), and so $-m+\nobreak\sum \abs{\C_i}\geq 0$. Hence,
after \cref{s:frozencohort} is executed, $I$ is nonnegative. Since other
steps increase $I$, and $I$ is bounded by $\abs{X}+4\abs{\F}$, it follows that \cref{s:frozencohort} must be executed at least once every $\abs{X}+4\abs{\F}$~steps.

After a cohort is disbanded in \cref{s:frozencohort}, no new cohort with the same banner can be
created, as banner is a floating element. Furthermore, since 
any cohort disbanded in \cref{s:frozencohort} is non-empty, 
that step can be executed at most $d$ times for any given value of the banner. So,
\cref{s:frozencohort} can be executed total of at most $d\abs{X}$ times.
In particular, the algorithm terminates after at most $d\abs{X}\bigl(\abs{X}+4\abs{\F}\bigr)$ steps.

Since \cref{s:newcohort} is executed unless $\F=\B$, when the algorithm terminates we have 
$\Th(S)\leq 2d-\Delta$ for all $S\in \F$. While it does not mean that the final
coloring $\chi$ takes only values $-1$ and $+1$, it does imply that no matter
how we round the remaining floating elements, the resulting coloring will
have discrepancy at most $2d-\Delta$.

\section{Creation of a new cohort}\label{sec:newcohort}
In this section we prove a claim made in \cref{s:newcohort}, namely, that if
that step can be executed, then there is a family $\Cb\subset \G$ 
all of whose sets contain a common element $b\in X$, and
that satisfies the right conditions for making a new cohort.
In what follows, we assume that none of the steps \ref{s:genbenign}
through \ref{s:perturb} can be executed, and that all the invariants hold.

\begin{lemma}\label{lem:cohortconstr}
For each cohort $\C_i$ we have
\begin{equation}\label{eq:cohortconstr}
  d\bigl(\abs{\C_i}-\abs{M_i}\bigr)<W+\sum_{S\in\C_i} \bigl(\Sz(S)-1\bigr).
\end{equation}
\end{lemma}
\begin{proof} By symmetry assume $\eps_i=+1$.

Because $D[S]=D[S']$ for $\{S,S'\}\in M_i$ by \cref{i:eqmatching} and because
\cref{s:exhaust} cannot be executed, we have
\begin{equation}\label{eq:matchedszbnd}
  \bigl(\Sz(S)-1\bigr)+\bigl(\Sz(S')-1\bigr)\geq d-2^{D[S]}-2^{D[S']}+1\qquad\text{for all }\{S,S'\}\in M_i.
\end{equation}

Let $S\in C_i\setminus M_i$ be any unmatched cohort set, and let $D=D[S]$ for brevity.
Because \cref{s:cohbenign} cannot be executed, we necessarily have $\Th(S)>2d-\Delta$,
which by \cref{lem:onesided_benign} implies that $\Th-(S)>2d-\Delta$,
and so 
\begin{align*}
  2\Sz(S)&\geq \Sz(S)+\Fl(S)&&\text{since }\abs{\Fl(S)}\leq \Sz(S)\\
         &=\Th-(S)+\chi(S)&&\text{because of \eqref{eq:Thn}}\\
         &> 2d-\Delta+\chi(S)=2d-\Delta-\chi(S)+2\chi(S)\\
         &\geq 2d-\Delta-\chi(S)+\alpha\Th+(S)-2^{D+1}\beta_{r_i}\alpha-2^{D+1} \Tw_{r_i}+2R_{D}
             &&\text{by \cref{i:Cchifreep}}\\
         &\geq 2d-2\Delta-2+2^{D+1}-2^{D+1}\beta_{r_i}\alpha-2^{D+1} \Tw_{r_i}+2R_{D}
             &&\text{by \cref{lem:other_chi_bound}}.
\end{align*}
Hence by \refconst{c}
\begin{equation*}
  \Sz(S)-1\geq  d-2^{D[S]+3}\Tw_{r_i}\qquad\text{for all }S\in \C_i\setminus V(M_i).
\end{equation*}

Let $L=4\Tw_{r_i}$.
Since \cref{s:match} cannot be executed, for each $D$ there is at most one set
$S\in \C_i$ satisfying $D[S]=D$ that is not in $M_i$. Also, $D[S]\leq L$
by \cref{i:Dbound}. Thus,
\[
  \sum_{S\in \C_i\setminus V(M_i)} (\Sz(S)-1) \geq \sum_{S\in \C_i\setminus V(M_i)} (d-2^{D[S]+3}\Tw_{r_i})
    \geq d\abs{\C_i\setminus V(M_i)}-2^{L+4}\Tw_{r_i}.
\]
Combining this with \eqref{eq:matchedszbnd} we obtain
\begin{equation}\label{eq:szbnd}
  \sum_{S\in \C_i} \bigl(\Sz(S)-1\bigr) \geq d\bigl(\abs{\C_i}-\abs{M_i}\bigr)-\sum_{S\in V(M_i)} 2^{D[S]}+\abs{M_i}-2^{L+4}\Tw_{r_i}.
\end{equation}
If $\abs{M_i}>2^{L+4}\Tw_{r_i}$, then the lemma follows from \cref{i:elim}. Otherwise,
$\abs{M_i}\leq 2^{L+4}\Tw_{r_i}$, hence $\abs{V(M_i)}\leq 2^{L+5}\Tw_{r_i}$, and so
\[
  \sum_{S\in V(M_i)} 2^{D[S]}+2^{L+4}\Tw_{r_i}\leq 2^L\cdot 2^{L+5}\Tw_{r_i}+2^{L+4}\Tw_{r_i}\leq 2^{2L+6}\Tw_{r_i}<W
\]
by \refconst{h}, and the lemma follows from \eqref{eq:szbnd}.
\end{proof}

\begin{lemma}\label{lem:Gszbnd} The sets in $\G$ satisfy
\begin{equation}\label{eq:Gszbnd}
  \sum_{S\in \G} (\Sz(S)-d)\leq 0,
\end{equation}
and the equality is possible only if there are no cohorts.
\end{lemma}
\begin{proof}
Let $F$ be the number of floating elements, and consider
\[
  dF\geq \sum_{S\in \F}\Sz(S)=\sum_{S\in \G}\Sz(S)+\sum_{S\in \B}\Sz(S)+\sum_{i=1}^m \sum_{S\in\C_i}\Sz(S)=\Sigma_1+\Sigma_2+\Sigma_3.
\]
By \cref{i:bannercnt} we have $\Sigma_2\geq \sum_i (W-\abs{\C_i})$. 
Bounding $\Sigma_3$ by the preceding lemma we obtain
\begin{equation}\label{eq:clu}
  dF\geq \sum_{S\in \G}\bigl(\Sz(S)-d\bigr)+d\abs{\G}+d\sum_{i=1}^m (\abs{\C_i}-\abs{M_i}).
\end{equation}
Furthermore, since the inequality in the preceding lemma is strict, 
the equality in \eqref{eq:clu} can hold only if there are no cohorts.

Finally, since \cref{s:perturb} cannot be executed, $F\leq \abs{\G}+\sum_i (\abs{\C_i}-\abs{M_i})$
and so \eqref{eq:Gszbnd} holds.
\end{proof}

Let
\begin{align*}
  B^+ &= \{ x \text{ floating}: \phantom{-}\chi(x)\geq 1-\alpha\},\\
  B^- &= \{ x \text{ floating}: -\chi(x)\geq 1-\alpha\}.
\end{align*}
The next lemma shows that nasty sets contain many nearly-frozen elements.
\begin{lemma}\label{lem:manyfrozen}
Suppose $S\in\G$ satisfies $\Sz(S)\leq d$. Then the following hold:
\begin{enumerate}
\item[$+$)] If $\Th+(S)>2d-\Delta$, then $\abs{B^-\cap S}\geq d/2$ and $\Sz(S)\geq d-\Delta/2-\Tw_{\Delta-1}$
\item[$-$)] If $\Th-(S)>2d-\Delta$, then $\abs{B^+\cap S}\geq d/2$ and $\Sz(S)\geq d-\Delta/2-\Tw_{\Delta-1}$
\end{enumerate}
\end{lemma}
\begin{proof}
Consider the case $\Th-(S)>2d-\Delta$. The other case is similar.

Define $r$ by $\Th-(S)=2d-r$.
From \eqref{eq:Thn} and \cref{i:Gchin} we deduce that
\[
  \tfrac{1}{2}\alpha \Th+(S)-\Tw_r\leq \chi(S)=\Sz(S)+\Fl(S)-\Th-(S),
\]
and since $\Th-(S)+\Th+(S)=2\Sz(S)$, this implies that
\begin{equation}\label{eq:pnegnew}
  -\Tw_r\leq (1-\alpha)\Sz(S)+\Fl(S)-\bigl(1-\tfrac{1}{2}\alpha\bigr)\Th-(S).
\end{equation}
We then use bounds $\Fl(S)\leq (1-\alpha)\abs{\Sz(S)}+\alpha\abs{B^+\cap S}$,
$\Th-(S)\geq 2d-\Delta$ and $\Sz(S)\leq d$ to obtain
\[
  -\Tw_r\leq -\alpha d+\alpha\abs{B^+\cap S}+(1-\tfrac{1}{2}\alpha)\Delta
\]
Hence $\abs{B^+\cap S}\geq d-\Tw_r/\alpha\geq d/2$ by \refconst{f}.

When combined with $\Fl(S)\leq \Sz(S)$, 
the inequality \eqref{eq:pnegnew} also implies that
\[
  (2-\alpha)\Sz(S)\geq \bigl(1-\tfrac{1}{2}\alpha\bigr)(2d-\Delta)-\Tw_r,
\]
and so $\Sz(S)\geq d-\Delta/2-\Tw_r$.
\end{proof}

We are now ready to demonstrate the claim made in \cref{s:newcohort}.

Define a \emph{charge} for a pair $(x,S)$ where $x\in B^+\cup B^-$ and $S\in \G$
by the following rule:
\[
  \Ch(x,S) = \begin{cases}
     \dfrac{\Sz(S)-d-1}{\abs{S\cap B^+}}&\text{if }x\in S\cap B^+\text{ and }\Sz(S)\leq d\text{ and }\Th-(S)>2d-\Delta,\\[2.3ex]
     \dfrac{\Sz(S)-d-1}{\abs{S\cap B^-}}&\text{if }x\in S\cap B^-\text{ and }\Sz(S)\leq d\text{ and }\Th+(S)>2d-\Delta,\\[2.3ex]   
     \dfrac{\Sz(S)-d}{4\abs{S \cap B^+}}&\text{if }x\in S\cap B^+\text{ and }\Sz(S)\geq d+1,\\[2.3ex]
     \dfrac{\Sz(S)-d}{4\abs{S \cap B^-}}&\text{if }x\in S\cap B^-\text{ and }\Sz(S)\geq d+1,\\[2.3ex]
     0&\text{otherwise}.
  \end{cases}
\]

Since \cref{s:genbenign} cannot be executed, $\Sz(S)\leq d$ implies that either
$\Th-(S)>2d-\Delta$ or $\Th+(S)>2d-\Delta$. Consider any $S\in \G$
satisfying $\Sz(S)\leq d$. By \cref{lem:manyfrozen} 
we have $S\cap (B_+\cup B_-)\neq \emptyset$, and so
\begin{equation}\label{eq:changeundercount}
  \sum_{x\in X} \Ch(x,S) <\Sz(S)-d.
\end{equation}
Inequality \eqref{eq:changeundercount} also clearly holds for $S\in \G$ satisfying $\Sz(S)>d$.
So \eqref{eq:changeundercount} holds for all $S\in \G$.

If $\G=\emptyset$, then the left-hand side of 
inequality \eqref{eq:Gszbnd} would be zero, implying that there are no cohorts,
and so $\F=\B$. Since $\F=\B$ contradicts the assumption that \cref{s:newcohort}
can be executed, we conclude that $\G$ is non-empty.
Hence the inequality \eqref{eq:changeundercount} 
and \cref{lem:Gszbnd} imply 
\[
  \sum_{\substack{x\in X\\S\in\G}}\Ch(x,S)<\sum_{S\in \G} \bigl(\Sz(S)-d\bigr)\leq 0.
\]
So, there is an element $b\in B^+\cup B^-$ such that
\begin{equation}\label{eq:negcharge}
  \sum_{S\in \G}\Ch(b,S)<0.
\end{equation}
Fix such a $b$, and assume that $b\in B^+$ (the case $b\in B^-$ is analogous). Define
\begin{align*}
  \Nb&\eqdef \{S\in \G : b\in S,\ \Sz(S)\leq d\text{ and }\Th-(S)>2d-\Delta\},\\
  \Pb&\eqdef \{S\in \G : b\in S,\ \Sz(S)\geq d+1\text{ and }\Th+(S)>2d-\Delta\}.
\end{align*}
From \eqref{eq:negcharge} and the definition of $\Ch(b,S)$ it follows that
$\Nb$ is non-empty.

Since \cref{s:round} is unable to round $b$, the set $\Pb$ must be non-empty as well. 
We next show that the contribution of a set from $\Pb$ to the sum \eqref{eq:negcharge}
is very large, whereas the contribution of a set from $\Nb$ is very small.
\begin{proposition}\label{prop:neg}
For each $S\in\Nb$ we have $\Ch(b,S)\geq -\frac{\Delta+2\Tw_{\Delta-1}+2}{d}$.
\end{proposition}
\begin{proposition}\label{prop:pos}
For each $S\in\Pb$ we have $\Ch(b,S)\geq \frac{1}{8+4\Delta}$.
\end{proposition}
\begin{proof}[Proof of \cref{prop:neg}]
This follows from $\Ch(b,S)=\bigl(\Sz(S)-d-1\bigr)/\abs{S\cap B^+}$,
and the bounds on $\Sz(S)$ and on $S\cap B^+$ from \cref{lem:manyfrozen}.
\end{proof}
\begin{proof}[Proof of \cref{prop:pos}]
We show that $\abs{S\cap B^+}$ is very small by bounding $\Fl(S)$ both from 
above and from below: 
\begin{align*}
\Fl(S)&=\Sz(S)+\chi(S)-\Th+(S)&&\text{by identity \eqref{eq:Thp}}\\
      &\leq \Sz(S)+\chi(S)-2d+\Delta,&&\text{since }S\in \Pb\\
      &\leq \Sz(S)-2d+\Delta,&&\text{by \cref{i:Gchip}}\\ 
\Fl(S)&\geq -\Sz(S)+(2-\alpha)\abs{S\cap B^+}.
\end{align*}
Since $\alpha\leq 1$, these two inequalities imply
\[
\abs{S\cap B^+}\leq 2\bigl(\Sz(S)-d\bigr)+\Delta,
\]
and so $\Ch(b,S)\geq \tfrac{1}{8}\cdot \frac{\Sz(S)-d}{\Sz(S)-d+\Delta/2}\geq 
\frac{1}{8(1+\Delta/2)}$.
\end{proof}

From the two propositions and from \eqref{eq:negcharge} we conclude that 
$\abs{\Nb}\geq  \frac{d}{\Delta+2\Tw_{\Delta-1}+2}\cdot\frac{1}{8+4\Delta} \geq \Delta W$ by \refconst{g}. By \cref{i:dangerbound}
we have $2d-\Delta<\Th-(S)\leq 2d$ for all $S\in \Nb$. Hence, by the pigeonhole
principle there is an $0\leq r<\Delta$ and a $\Cb\subset \Nb$ of size $\abs{\Cb}=W$ 
such that $\Th+(S)=2d-r$ for all $S\in\Cb$. This completes the proof of the claim
made in \cref{s:newcohort}.

\bibliographystyle{plain}
\bibliography{beck_fiala_improvement}

\end{document}